\documentclass[11pt, reqno]{amsart}

\calclayout

\usepackage{mathtools}
\usepackage{amsthm}
\usepackage{amsmath}
\usepackage{amssymb}
\usepackage{bm} 
\usepackage{thmtools}
\usepackage{stmaryrd}
\usepackage{enumerate}
\usepackage{float}
\usepackage{cite}
\usepackage[hidelinks]{hyperref}
\usepackage{cleveref}
\usepackage[font=small,labelfont=bf]{caption}
\usepackage[dvipsnames]{xcolor}
\usepackage[all]{xy}
\usepackage[normalem]{ulem}
\usepackage[UKenglish]{babel}

\theoremstyle{definition}
\newtheorem{defn}{Definition}[section]

\newtheorem{prop}[defn]{Proposition}
\newtheorem{rmk}[defn]{Remark}
\newtheorem{eg}[defn]{Example}
\newtheorem{thrm}[defn]{Theorem}
\newtheorem{cor}[defn]{Corollary}

\newcommand{\F}{\mathbb{F}}
\newcommand{\N}{\mathbb{N}}
\newcommand{\Aut}{\mathrm{Aut}}
\newcommand{\Perm}{\mathrm{Perm}}
\newcommand{\Gal}{\mathrm{Gal}}
\newcommand{\Map}{\mathrm{Map}}
\newcommand{\Norm}{\mathrm{Norm}}
\newcommand{\Hol}{\mathrm{Hol}}

\newcommand{\Stab}{\mathrm{Stab}}
\newcommand{\GL}{\mathrm{GL}}
\newcommand{\id}{\mathrm{id}}
\newcommand{\into}{\hookrightarrow}

\newcommand{\acts}[1]{\,^{#1}}
\newcommand{\barg}{\overline{g}}

\begin{document}

\title[Almost Classical Skew Bracoids]{Almost Classical Skew Bracoids}

\author{Isabel Martin-Lyons}
\address{School of Computer Science and Mathematics \\ Keele University \\ Staffordshire \\ ST5 5BG \\ UK}
\email{I.D.Martin-Lyons@keele.ac.uk}

\subjclass[2020]{12F10; 
16T05; 
16T25 
20N99; 
}

\keywords{Skew bracoids, skew braces, Hopf-Galois structures, Yang-Baxter equation}

\begin{abstract}
We investigate two sub-classes of skew bracoids, the first consists of those we term \textit{almost a brace}, meaning the multiplicative group decomposes as a certain semi-direct product, and then those that are \textit{almost classical}, which additionally specifies the relationship between the multiplicative group and the additive. Skew bracoids with these properties have applications in Hopf-Galois theory, in particular for questions concerning the Hopf-Galois correspondence, and can also yield solutions to the set-theoretic Yang-Baxter equation. We use this skew bracoid perspective to give a new construction building on the induced Hopf-Galois structures of Crespo, Rio and Vela, recover a result of Greither and Pareigis on the Hopf-Galois correspondence, and examine the solutions that arise from skew bracoids, in particular where more than one solution may be drawn from a single skew bracoid.
\end{abstract}

\maketitle

\section{Introduction} \label{intro}

The skew bracoid was introduced in \cite{mltruoid}, joining a myriad of generalisations of Guarnieri and Vendramin's skew brace \cite{guaven} (see for example \cite{ccssemi}, \cite{weakbrace}, \cite{drnear}), which is itself a generalisation of Rump's brace \cite{rbrace}. A \textit{skew brace} $(G,\star,\cdot)$ has two group structures, $\star$ \textit{additive} and $\cdot$ \textit{multiplicative}, on the same underlying set. The interaction between these structures is governed by the \textit{skew brace relation}: 
\[g\cdot (h_1 \star h_2) = (g\cdot h_1) \star g^{-\star} \star (g\cdot h_2) \,\,\,\,\text{for all }g,h_1,h_2\in G,\] 
where $g^{-\star}$ denotes the inverse of $g$ with respect to $\star$. A \textit{skew bracoid} $(G,\cdot,N,\star,\odot)$ comprises two groups, $(N,\star)$ \textit{additive} and $(G,\cdot)$ \textit{multiplicative}, that relate via a transitive action $\odot$ of the latter on the former, satisfying an analogous compatibility relation, the \textit{skew bracoid relation}:
\[g\odot (\eta_1 \star \eta_2) = (g\odot \eta_1) \star (g\odot e_N)^{-1} \star (g\odot \eta_2) \,\,\,\,\text{for all }g \in G\text{ and all }\eta_1,\eta_2\in N.\] 
For brevity, we commonly write $(G,N,\odot)$, or even $(G,N)$, for $(G,\cdot,N,\star,\odot)$, and suppress $\star$ and $\cdot$ where possible. 

To see the skew bracoid as a generalisation of the skew brace, notice that any given skew brace $(G,\star,\cdot)$ may be considered a skew bracoid $(G,\cdot,G,\star,\odot)$, taking $\odot$ to be the operation $\cdot$. Conversely, for a skew bracoid $(G,N)$ in which $S=\Stab_G(e_N)$ is trivial we may transfer the operation from one group onto the other to produce a bona fide skew brace. We say that such a skew bracoid is \textit{essentially a skew brace} \cite[Example 2.2]{mltruoid}. 

The brace, and later skew brace, were developed primarily for the study of solutions to the set-theoretic Yang-Baxter equation, which was first suggested by Drinfeld over thirty years ago \cite{drinfeld}. A \textit{solution to the set-theoretic Yang-Baxter equation} (hereafter simply a \textit{solution}) consists of a set $G$ and a map $\bm{r}:G\times G\to G\times G$ satisfying
\[(\bm{r}\times \id)(\id\times \bm{r})(\bm{r}\times \id)=(\id\times \bm{r})(\bm{r}\times \id)(\id\times \bm{r})\]
as functions on $G\times G\times G$. A large family of skew bracoids have also been connected to solutions \cite{ckmtybe}. To produce a solution in the manner the authors outline, the skew bracoid $(G,N)$ must \textit{contain a brace}, meaning that there is a subgroup $H$ of $G$ for which $(H,N)$ is essentially a skew brace or equivalently that $G$ has the exact factorisation $HS$ where $S=\Stab_G(e_N)$ \cite{ckmtybe}. It is therefore of interest to study skew bracoids in which the multiplicative group admits such an exact factorisation. In this work, we investigate this case and moreover, that in which $G$ decomposes as a semi-direct product $H\rtimes S$.

This exploration is also motivated by Hopf-Galois theory, which has been linked to both the skew brace and the skew bracoid. To outline this connection, we must go back to the ground-breaking work of Greither and Pareigis, which classifies Hopf-Galois structures on finite separable extensions of fields via permutation groups \cite{gp}. A result of Byott reduces the problem from this permutation setting to that of the \textit{holomorph of a group} (see \Cref{sechol}) \cite{btrans}; Hopf-Galois structures on Galois extensions and skew braces are then linked by their mutual connection to the regular subgroups of the holomorph \cite{bracehgs}\cite[Appendix A]{svb}. Stefanello and Trappeniers have now refined the connection to express it as a bijection and align the natural substructures across the correspondence \cite{sttwistpub}; it is this formulation that is extended to the separable, skew bracoid case in \cite{mltruoid}.

One of the central questions of Hopf-Galois theory is that of the Hopf-Galois correspondence. Given a Hopf-Galois structure on some finite separable extension of fields $L/K$, there is a natural notion of the fix of $L$ by a K-Hopf sub-algebra, returning an intermediate field of $L/K$. This fix map is always injective and inclusion-reversing but unlike in classical Galois theory, it is not necessarily surjective. The question of surjectivity has seen a lot of attention in the Galois case (see \cite{chcircstab}, \cite{kktucorresp}, \cite{chhgc}, \cite{sttwistpub}), while \cite{crvhgc} is perhaps the only focused study of the separable case in recent years. However, there is a plentiful class of examples for which the Hopf-Galois correspondence is surjective in the separable case, identified by Greither and Pareigis in their 1987 paper \cite{gp} and classified for various degrees in for example \cite{kohlradext}, \cite{byalmost}, \cite{cslowdeg}, \cite{darpq}. Writing $E$ for the Galois closure of $L/K$ and $S$ for $\Gal(E/L)$, we say that $L/K$ is \textit{almost classically Galois} or simply \textit{almost classical} if there is an $H$ for which the Galois group $G=\Gal(E/K)$ decomposes as a semi-direct product $H\rtimes S$. With an explicit description, Greither and Pareigis prove that these extensions admit at least one Hopf-Galois structure for which the Hopf-Galois correspondence is surjective \cite[Theorem 5.2]{gp}, it is these structures that Kohl later terms themselves \textit{almost classical} \cite{kohlradext}.

As alluded to, here we discuss how almost classical manifests in a skew bracoid, both at an extension and a structural level. In \Cref{secalcla}, we give the main definitions and explore the properties of such skew bracoids in their own right, including an examination of the all-important $\gamma$-function \cite[Theorem 2.8]{mltruoid}. In \Cref{sechg}, we confirm the connection to the Hopf-Galois setting, and use this skew braciod description to recover the result of Greither and Pareigis on the Hopf-Galois correspondence for almost classical structures \cite[Theorem 5.2]{gp}. It is our hope that this may work as a prototype for the study of the Hopf-Galois correspondence using skew bracoids, which has seen such success in the skew brace case.

The hypothesis that $\Gal(E/K)$ is a semi-direct product is also employed by Crespo, Rio and Vela in their construction of \textit{induced} Hopf-Galois structures \cite[Theorem 3]{crvind}. This uses a Hopf-Galois structure on $L/K$ and $E/L$ to induce one on $E/K$, or, in the language of skew bracoids, this is constructing a skew brace $(G,\star,\cdot)$ from a skew bracoid $(G,\cdot,G/S,\star,\odot)$ and a skew brace $(S,\star, \cdot)$. In \Cref{inducesec}, we give a more general construction, taking two related skew bracoids to induce a third, the case in which the second is essentially a skew brace gives the existing construction.

In \Cref{sechol}, we consider these properties from the holomorph, which is already fairly well understood at an extension level but we provide a clean description at the structural level. To do this, we exploit the correspondence between Hopf-Galois structures and subgroups of the holomorph, via \cite{gp} and \cite{btrans}, and skew bracoids and subgroups of the holomorph, via \cite[Theorem 2.8]{mltruoid}. However, we note that the existing route from Hopf-Galois structure to skew bracoid does not necessarily yield the same skew bracoid and the one obtained by first passing to the holomorph, we address this in \Cref{sectrans}.  

In \Cref{secsol}, we collate and interpret the results of the proceeding sections as they pertain to solutions. We also move towards addressing the question of how the choice of $H$, if there are multiple candidates, affects the resulting solution.

\section{Almost Classical Skew Bracoids}
\label{secalcla} 

In this section we introduce our main definitions, provide examples and discuss various properties of these objects.

\begin{defn}
Let $(G,N)$ be a skew bracoid and $S=\Stab_G(e_N)$. We say that $(G,N)$ is \textit{almost a brace} (\textit{with respect to $H$}) if and only if $S$ has a normal complement $H$ in $G$, so that $G$ decomposes as a semi-direct product $H\rtimes S$. 
\end{defn}

\begin{rmk}
    In the language of \cite{ckmtybe}, this is saying $(G,N)$ contains a brace with an additional normality condition. Hence each skew bracoid that is almost a brace is connected to a \textit{semi-brace} (see \cite{ccssemi}), and can therefore be used to produce a solution to the set-theoretic Yang-Baxter equation as described in \cite{ckmtybe}. The almost a brace case is also explored by Castelli from a semi-brace perspective in \cite[Section 3]{castacsemi}.
\end{rmk}

As with the terminology ``contain a brace", the omission of ``skew" in ``almost a brace" is not intended to imply that the additive structure is abelian. More substantively, if we have a skew bracoid that is almost a brace it is automatic that $(H,N)$ is a sub-skew bracoid of $(G,N)$, in fact to contain a brace is enough to ensure this \cite{ckmtybe}. Clearly $H$ must be a subgroup of $G$ and $H$ is transitive on $N$ as \[H\odot e_N= HS\odot e_N=G\odot e_N = N.\]
From the exact factorisation of $G$ we also have that $H\cap S= \{e_G\}$, so $\Stab_H(e_N)=\{e_G\}$, which means $H$ acts regularly on $N$ and $(H,N)$ is essentially a skew brace. To elucidate further, this gives that the map  $h\mapsto h\odot e_N$ is a bijection between $H$ and $N$, which can be used to transfer the operation in $N$ onto $H$ (say) to give a skew brace structure to $H$.

Once viewed as a skew brace, it is reasonable to ask if the group operations in $(H,N)$ coincide, that is if the skew brace is \textit{trivial}.

\begin{defn}
We say that a skew bracoid $(G,N)$ is \textit{almost classical} if and only if it is almost a brace with respect to some $H$, for which the sub-skew bracoid $(H,N)$ is trivial when thought of as a skew brace.
\end{defn}

Explicitly, this is the same as saying
\begin{equation}\label{niceop}
(h_1 \odot e_N)(h_2\odot e_N)= h_1h_2\odot e_N
\end{equation}
for all $h_1,h_2\in H$. It is then easy to see that for any $h\in H$ we have
\begin{equation}\label{niceinv}
(h\odot e_N)^{-1}=h^{-1}\odot e_N,
\end{equation}
which is unusual in a skew bracoid, and indeed may not hold for some $g\in G \setminus H$.

Henceforth, we will use the phrase \textit{essentially trivial} to mean that the skew bracoid is essentially a skew brace and is trivial when thought of as such.

\begin{eg}
Every skew brace is almost a brace. Let $(N,N)$ be a skew brace thought of as a skew bracoid, then $S=\Stab_N(e_N)$ is trivial so it has the full group $N$ as a normal complement. Further, $(N,N)$ is almost classical precisely when it is trivial.
\end{eg}

\begin{eg}
We can construct skew bracoids that are almost a brace using the asymmetric product of Catino, Colazzo and Stefanelli \cite[Theorem 3]{ccsasym} and the partial quotienting procedure of \cite[Proposition 2.4]{mltruoid}. 

Let $(H,+_H,\cdot_H)$ and $(S,+_S,\cdot_S)$ be braces, so that $(H,+_H)$ and $(S,+_S)$ are abelian groups, and suppose we have a symmetric $2$-cocycle on $(H,+_H)$ taking values in $S$ and a homomorphism from $(H,\cdot_H)$ to the brace automorphisms of $(S,+_S,\cdot_S)$ satisfying the compatibility relation given in \cite[Theorem 3]{ccsasym}. One can give a brace structure to $H\times S$ via the asymmetric product, denoted $H\rtimes_\circ S$, this brace has multiplicative group isomorphic to $(H, \cdot_H)\rtimes (S,\cdot_S)$ and additive group isomorphic to $(H,+_H)\times (S,+_S)$.

As noted in \cite{bcjoasym}, the set $\{e_H\} \times S$ is then a (strong) left ideal of $H\rtimes_\circ S$, so following \cite[Proposition 2.4]{mltruoid} we may take the additive quotient to produce a skew bracoid of the form $(H \times S, H)$. As always, the stabiliser of the additive identity is the subgroup by which we took the quotient, $\{e_H\} \times S$, which has a normal complement, $H \times \{e_S\}$, so the skew bracoid is almost a brace.
\end{eg}

\begin{eg}
\label{d2ncd} 
Let $d,n \in \N$ be such that $d\,|\, n$, then we may take $G=\langle r, s \;|\; r^n=s^2=e, \; srs = r^{-1} \rangle \cong D_{2n}$, $N=\langle \eta \rangle \cong C_d$ and have $G$ act on $N$ via \[ r^is^j\odot \eta^k = \eta^{i+(-1)^jk}.\]
As discussed in \cite[Example 2.3]{mltruoid}, this describes a family of skew bracoids. By inspection of the action we see that $S=\Stab_G(e_N)=\langle r^{d}, s \rangle$, we investigate when this has a complement in $G$. Note that any such complement has order $d$. 

Skew bracoids in this family are \textit{reduced}, that is the action $\odot$ is faithful, precisely when $d=n$ \cite[Example 2.16]{mltruoid}. In this case $S= \langle s \rangle$, which has $R=\langle r \rangle$ as a normal complement and since $r^ir^j\odot e_N=\eta^{i+j}=(r^i\odot e_N)(r^j\odot e_N)$, skew bracoids of the form $(D_{2n},C_n)$ are almost classical. If we take $n$ to be even then $S$ has $H= \langle r^2,rs\rangle$ as an additional normal complement, but here $H$ is not isomorphic to $N$ so $(D_{2n},C_n)$ is merely almost a brace with respect to $H$.
 
Allowing $d$ to differ from $n$, we see a surprising variety of behaviour. Taking $n=12$ and $d=4$ for example, we have that $(G,N)$ is almost classical with respect to $R=\langle r^3 \rangle$ but also contains a brace with respect to $H_c=\langle r^6, r^cs\rangle \cong D_4$ for $c\in \{1,3,5\}$. Instead taking $n=12$ and $d=6$, we have lost the cyclic complement but the skew bracoid still contains a brace with respect to $H_c=\langle r^4, r^cs\rangle \cong D_6$ for $c\in \{1,3\}$. Finally, taking $n=9$ and $d=3$ we have a skew bracoid that does not even contain a brace.
\end{eg}

Each of the non-reduced examples is a member of an infinite family of skew bracoids with those attributes, in particular, skew bracoids $(D_{2n},C_d)$ for which $n$ is odd and $(d,\frac{n}{d})>1$ do not contain a brace. In the reduced case we believe that non-examples are less plentiful, though we have an occurrence of this due to Darlington \cite{darpq}, highlighted in \cite[Example 2.6]{ckmtybe}.

Lastly, to see that to be almost a brace is a stronger condition than to contain a brace, even in the reduced case, we turn to \cite[Theorem 3.6]{bysol}.

\begin{eg}
    As discussed in \cite[Example 2.22]{mltruoid}, the result of Byott implies the existence of a skew bracoid whose additive group is elementary abelian of order $8$ and multiplicative group is isomorphic to $\GL_3(\F_2)$. We saw in \cite[Example 2.5]{ckmtybe} that this skew bracoid contains a brace with respect to any Sylow-2 subgroup of the multiplicative group, but as $\GL_3(\F_2)$ is simple, these subgroups are not normal and therefore the skew bracoid is not almost a brace.
\end{eg}

We now turn our attention to the $\gamma$-function, this will shed more light on the inner-workings of these skew bracoids and will be central to our upcoming correspondences.

\begin{prop}
\label{acgamma}
Let $(G,N)$ be a skew bracoid that is almost classical with respect to $H$ and write $S$ for $\Stab_G(e_N)$. For all $s_1\in S$ and all $h_1,h_2\in H$, we have that 
\begin{equation*}
    \acts{\gamma(h_1s_1)}(h_2 \odot e_N)= s_1h_2s_1^{-1} \odot e_N.
\end{equation*}

Note that since $G=HS$ and $H$ is transitive on $N$, this completely determines the $\gamma$-function. We will repeatedly describe $G$ in this way.
\end{prop}

\begin{proof}
Let $s_1\in S$ and $h_1,h_2\in H$. Then
\begin{align*}
    \acts{\gamma(h_1s_1)}(h_2\odot e_N) &= (h_1s_1 \odot e_N)^{-1}(h_1s_1 \odot (h_2\odot e_N)) \\
                        &= (h_1\odot e_N)^{-1}(h_1s_1h_2 \odot e_N)    \\
                        &= (h_1^{-1} \odot e_N)(h_1s_1h_2s_1^{-1} \odot e_N) &&\text{ by (\ref{niceinv})}  \\
                        &= s_1h_2s_1^{-1}\odot e_N &&\text{ by (\ref{niceop}).}\qedhere
\end{align*}
\end{proof}

In particular, $\gamma(h)$ is the identity on $N$ for all $h\in H$. In fact this characterises almost classical skew bracoids within the class of skew bracoids that are almost a brace.

\begin{prop} \label{gamker}
    Let $(G,N)$ be a skew bracoid that is almost a brace with respect to $H$. Then $H\subseteq \ker(\gamma)$ if and only if $(G,N)$ is almost classical with respect to $H$.
\end{prop}

\begin{proof}
    The statement $H\subseteq \ker(\gamma)$ means precisely that for all $h_1,h_2\in H$
    \begin{align*}
        && h_2 \odot e_N &= \acts{\gamma(h_1^{-1})}(h_2\odot e_N)          && \\
        &&              &= (h_1^{-1} \odot e_N)^{-1}(h_1^{-1}h_2\odot e_N) &&\text{by definition of $\gamma$}\\
        &&              &= h_1^{-1}\odot ((h_1\odot e_N)(h_2\odot e_N))    &&\text{by the skew bracoid relation.}
    \end{align*}    
    This occurs if and only if $h_1h_2\odot e_N = (h_1\odot e_N)(h_2\odot e_N)$ for all $h_1,h_2 \in H$, which is exactly \Cref{niceop}, the condition for $(G,N)$ to be almost classical with respect to $H$.
\end{proof}

With this description of the $\gamma$-function we obtain a strong result concerning the left ideals of almost classical skew bracoids.

\begin{prop}
\label{hgcsb}
Let $(G,N)$ be an almost classical skew braciod and $S=\Stab_G(e_N)$. If $G'$ is a subgroup of $G$ containing $S$, then $G'\odot e_N$ is a left ideal of $(G,N)$.
\end{prop}

\begin{proof}
It is enough to show that $G'\odot e_N$ is closed under $\gamma(G)$, as this implies the subgroup condition \cite[Proposition 3.11]{mltruoid}.

Let $g\in G$ and $g'\in G'$, then write $g=hs, g'=h's'$ with $h,h'\in H$ and $s,s'\in S$. Now
\begin{align*}
    \acts{\gamma(hs)}(h's'\odot e_N) &= \acts{\gamma(hs)}(h'\odot e_N)&& \\
                                &= sh's^{-1}\odot e_N &&\text{ by \Cref{acgamma}.}
\end{align*} 
Notice that $s,s^{-1}\in S\subseteq G'$ and $h'=(h's')(s')^{-1} \in G'$, so $sh's^{-1}\odot e_N \in G'\odot e_N$ as required.
\end{proof}

\section{Induced Skew Bracoids}\label{inducesec}

Using skew bracoids that are almost a brace, we give a construction that points towards a semi-direct product of skew bracoids. This builds on the work of Crespo, Rio and Vela on induced Hopf-Galois structures \cite{crvind}, the relationship between our construction and theirs is discussed in \Cref{sechg}.

\begin{thrm}
\label{induce}
Let $(G, N, \odot_N)$ be a skew bracoid that is almost a brace with respect to $H$ and write $S$ for $\Stab_G(e_N)$. Suppose we also have a skew bracoid $(S,M,\odot_M)$. 
We may take the (external) direct product of $N$ and $M$ and, writing $\pi$ for the natural projection of $G$ onto $S$, define
\[g\odot (\eta,\mu) := (g\odot_N \eta, \pi(g)\odot_M \mu)\] for $g\in G$ and $(\eta,\mu) \in N\times M$. This $\odot$ is then an action, under which $(G,N\times M,\odot)$ is a skew bracoid.
\end{thrm}

\begin{proof}
We can see that $\odot$ is an action from the fact that $\odot_N$ and $\odot_M$ are, and that $\pi$ is a homomorphism: for all $g_1,g_2\in G$, $\eta\in N$ and $\mu\in M$ we have
\begin{align*}
    g_1g_2\odot (\eta, \mu)&= (g_1g_2 \odot_N \eta, \pi(g_1g_2)\odot_M \mu) \\
                    &= (g_1\odot_N (g_2\odot_N \eta),\pi(g_1)\odot_M(\pi(g_2)\odot_M \mu)) \\
                    &= g_1\odot (g_2 \odot (\eta,\mu)),\\
    e_G\odot (\eta, \mu)&= (e_G\odot_N \eta, \pi(e_G)\odot_M \mu)\\
                    &= (\eta, e_S\odot_M \mu)                   \\
                    &= (\eta, \mu).
\end{align*} 
Recall $G=HS$, and observe that for $h\in H$ and $s\in S$ we have 
\begin{equation}\label{indact}
hs \odot (e_N,e_M)= (h\odot_N e_N, s\odot_M e_N);    
\end{equation}
then since $H$ is regular on $N$ via $\odot_N$ and $S$ is transitive on $M$ via $\odot_M$, the action $\odot$ of $G$ on $N\times M$ is also transitive.

It remains to verify the skew bracoid relation. Let $(\eta_1,\mu_1),(\eta_2,\mu_2)\in N\times M$ and $g\in G$, then using that the skew bracoid relation holds in $(G,N)$ and $(S,M)$ we have
\begin{align*}
    g \;\odot &\;((\eta_1,\mu_1)(\eta_2,\mu_2))                                     \\ 
                &= g\odot (\eta_1\eta_2,\mu_1\mu_2)                                 \\
                &= (g\odot_N (\eta_1\eta_2),\pi(g)\odot_M (\mu_1\mu_2))             \\
                &= ((g\odot_N\eta_1)(g\odot_N e_N)^{-1}(g\odot_N \eta_2),           \\
                        & \hspace{50px} (\pi(g)\odot_M\mu_1)(\pi(g)\odot_M e_M)^{-1}(\pi(g)\odot_M \mu_2))                                                     \\
                &= (g\odot_N\eta_1,\pi(g)\odot_M \mu_1)(g\odot_N e_N,\pi(g)\odot_M e_M)^{-1}(g\odot_N \eta_2,\pi(g)\odot_M\mu_2)                                               \\
                &= (g\odot (\eta_1,\mu_1))(g\odot (e_N,e_M))^{-1}(g\odot (\eta_2,\mu_2)),
\end{align*}
as required.
\end{proof}

Note that the stabiliser, $\Stab_G(e_N,e_M)$, in the induced skew bracoid coincides with $\Stab_S(e_M)$ in $(S,M)$. To see this let $h\in H$ and $s\in S$, then from (\ref{indact}) and the fact that $H$ acts regularly on $N$, we have that $hs\in \Stab_G(e_N,e_M)$ if and only if $h=e_G$ and $s\in \Stab_S(e_M)$, i.e if and only if $hs=s\in \Stab_S(e_M)$.

\begin{eg}
    Consider the skew bracoid $(G,N)\cong (D_{24},C_4)$ from \Cref{d2ncd}, here $S=\langle r^4,s\rangle$ which has $R=\langle r^3\rangle$ as a normal complement. Since $S\cong D_6$ we can take $M$ to be cyclic of order $3$ say, with generator $\mu$, and use $(S,M)\cong (D_6,C_3)$ from the same family. Noting that $\pi(r^is^j)= \pi(r^{4i-3i}s^j)=\pi(r^{4i}s^jr^{-3(-1)^ji})=r^{4i}s^j$, the action of $G$ on $N\times M$ is then given by
    \[ r^is^j \odot \eta^k\mu^\ell = \eta^{i+(-1)^jk}\mu^{4i+(-1)^j\ell}=\eta^{i+(-1)^jk}\mu^{i+(-1)^j\ell}.\]
    Notice that $N\times M$ is cyclic of order $12$ with generator $\eta\mu$, and $r^is^j \odot (\eta\mu)^k=(\eta\mu)^{i+(-1)^jk},$
    so our induced skew bracoid is of the form $(D_{24},C_{12})$.
\end{eg}

It is natural to consider which of our properties are invariant under this construction. Taking two (compatible) almost classical skew bracoids, the induced skew bracoid is not necessarily itself almost classical. This is perhaps easiest to see by taking an almost classical skew bracoid $(G,N)$ with the trivial skew brace on $S$, which we recall is almost classical. The resulting skew bracoid has multiplicative group $G\cong H\rtimes S$ and additive group $N\times S \cong H\times S$, hence whenever the action of $S$ on $H$ in $H\rtimes S$ is non-trivial, the skew bracoid $(G,N\times S)$ will be non-trivial as a skew brace and therefore not almost classical. However, we can answer in the affirmative for skew bracoids that are almost a brace or contain a brace.

\begin{prop}
\label{respind}
    Let $(G,N)$ be a skew bracoid that is almost a brace with respect to $H$ and write $S$ for $\Stab_G(e_N)$. Suppose $(S,M)$ is a skew bracoid that contains a brace (resp. is almost a brace) with respect to $R$. Then the induced skew bracoid $(G,N\times M)$ contains a brace (resp. is almost a brace) with respect to $HR$.
\end{prop}

\begin{proof}
    Write $S_M$ for $\Stab_S(e_M)$ which we have shown to coincide with $\Stab_G(e_N,e_M)$. Now suppose that $(S,M)$ contains a brace with respect to $R$, so that $S$ has the exact factorisation $RS_M$. Then $G=HS=HRS_M$, and the factorisation $(HR)S_M$ is exact because for $h\in H$ and $r\in R$, $hr\odot (e_N,e_M)= (h\odot_N e_N, r\odot_M e_M) = (e_N,e_M)$ implies $h=r=e_G$ by regularity, so $HR$ has trivial intersection with $S_M$.

    Suppose additionally that $R$ is normal in $S$ so that $(S,M)$ is almost a brace with respect to $R$, we claim that $HR$ is then normal in $G$. Let $h_1,h_2\in H$, $s\in S$ and $r\in R$ then
    \begin{align*}
        h_1s\cdot h_2r \cdot s^{-1}h_1^{-1} &= h_1sh_2 s^{-1} s rs^{-1} h_1^{-1} \\
                            &= h_1\cdot sh_2s^{-1} \cdot (srs^{-1})h_1^{-1}(sr^{-1}s^{-1}) \cdot srs^{-1}   \\
                            &\in HR,
    \end{align*}
    using the normality of $H$ in $G$ and $R$ in $S$. Hence $(G,N\times M)$ is almost a brace with respect to $HR$.
\end{proof}

With this construction we can say that all skew bracoids that are almost a brace occur as an additive quotient of a skew brace by a strong left ideal as outlined in \cite[Proposition 2.4]{mltruoid}. Notice that the following holds without imposing the condition that the skew bracoid is reduced.

\begin{prop}
Let $(G,N)$ be a skew bracoid that is almost a brace. There is a further operation $\star$ on $G$ such that $(G,\star,\cdot)$ is a skew brace which contains $S$ as a strong left ideal giving $(G,N)$ isomorphic to $(G,G/S)$ in the sense of \cite[Section 4]{mltruoid}.
\end{prop}

\begin{proof}
The skew bracoid $(G,N)$ is almost a brace so $G=HS\cong H\rtimes S$ for $S=\Stab_G(e_N)$ and $H$ some normal complement to $S$. We may take the trivial skew brace on $S$ and apply \Cref{induce} to obtain a skew bracoid $(G,N\times S)$. The stabiliser $\Stab_G(e_N,e_S)$ then coincides with $\Stab_S(e_S)=\{e_S\}$, so $(G,N\times S)$ is essentially a skew brace and the map $hs\mapsto hs\odot (e_N,e_S)=(h\odot e_N,s)$ between $G$ and $N\times S$ is a bijection. We use this to implicitly define an operation $\star$ on $G$ by
\begin{align}
    (h_1s_1\star h_2s_2)\odot (e_N,e_S)&= (h_1\odot_N e_N,s_1)(h_2\odot_N e_N,s_2) \nonumber \\
                                    &= ((h_1\odot_N e_N)(h_2\odot_N e_N),s_1s_2)\label{sbop}
\end{align} 
for all $h_1,h_2\in H$ and all $s_1,s_2\in S$, so that $(G,\star,\cdot)$ is a skew brace with $(G,\star)\cong N\times S$.

As the induced operation on $N\times S$ is the direct product, $S$ is a normal subgroup of $G$ under $\star$. Also, for $h_1\in H$ and $s_1,s_2 \in S$
\begin{align*}
    \acts{\gamma(h_1s_1)}s_2 &= (h_1s_1)^{-\star}\star(h_1s_1\cdot s_2) &&    \\
                             &= (h_1s_1)^{-\star}\star(h_1s_1)\star s_2 &&\text{from (\ref{indact}) and (\ref{sbop}) with $h_2=e_G$}    \\
                             &= s_2,
\end{align*}
so $S$ is closed under the $\gamma$-function in $(G,\star,\cdot)$, in fact $\gamma(G)$ is trivial on $S$. Hence, $S$ is a strong left ideal of $(G,\star,\cdot)$ and we may follow \cite[Proposition 2.4]{mltruoid} to form a skew bracoid $(G,G/S)$ as expected. Recall that the action of $G$ on $G/S$ here is left translation of cosets via $\cdot$.

To provide a skew bracoid isomorphism between $(G,N)$ and $ (G,G/S)$ we take identity map on $G$, which is a group isomorphism and note $\Stab(e_N)=S=\Stab(e_{G/S})$. This gives rise to the map $g\odot e_N \mapsto gS$ between $N$ and $G/S\cong(N\times S)/S$, which is an isomorphism under $\star$ by construction. Hence $(G,N)\cong (G,G/S)$ as skew bracoids.
\end{proof}

\section{The view from the Holomorph}
\label{sechol}

We know from \cite[Theorem 2.8]{mltruoid} that there is a correspondence between skew bracoids $(G,N)$ and transitive subgroups $A$ of $\Hol(N)$. This is via the homomorphism $\lambda_\odot:G\to \Perm(N)$, given by $\lambda_\odot(g)[\eta]=g\odot \eta$ for $g\in G$ and $\eta\in N$, whose image is contained in $\Hol(N)$. There, $\Hol(N)$ was viewed as the normaliser of $\lambda_\star(N)$ in $\Perm(N)$, we wish to consider this correspondence using the abstract formulation of $\Hol(N)$. That is $N\rtimes \Aut(N)$, where $(\eta_1,\theta_1)(\eta_2,\theta_2)=(\eta_1\theta_1(\eta_2),\theta_1\theta_2)$, which has a natural action on $N$ given by $(\eta,\theta)\mu=\eta\theta(\mu)$. Note that for all $\eta\in N$ and all $g\in G$, 
\begin{equation}\label{biggamma}
\lambda_\odot(g)[\eta]= g\odot \eta= (g\odot e_N)\acts{\gamma(g)}\eta= (g\odot e_N,\gamma(g))\eta,
\end{equation}
so we may consider $\lambda_\odot$ to be the map taking $g$ to $(g\odot e_N,\gamma(g))$.

For simplicity, for the bulk of this section skew bracoids are assumed to be reduced meaning that $\lambda_\odot$ is an injection. To this end, we have the following:

\begin{prop}
\label{respred}
Let $(G,N)$ be a skew bracoid, $S=\Stab_G(e_N)$, and $K=\ker(\lambda_\odot)$, so that $(G/K,N)$ is the \textit{reduced form} of $(G,N)$ (see \cite[Proposition 2.18]{mltruoid}). If $(G,N)$ contains a brace (resp. is almost a brace, is almost classical) then $(G/K,N)$ contains a brace (resp. is almost a brace, is almost classical).
\end{prop}

\begin{proof}
Suppose $(G,N)$ is contains a brace $(H,N)$, so that $G= HS$, then as $K\subseteq S$ we have $G/K= ( H/K)( S/K)$. As $\Stab_{G/K}(e_N)=S/K$, it follows immediately that $(G/K,N)$ contains a brace. If $H$ is additionally normal in $G$ so that $(G,N)$ is almost a brace with respect to $H$, then $H/K$ is also normal in $G/K$ so $(G/K,N)$ is almost a brace with respect to $H/K$. Finally, if $(G,N)$ is almost classical with respect to $H$, then since $H\cap K$ is trivial, $H/K\subseteq G/K$ is isomorphic to $H$ itself, and thus $(H/K,N)\cong (H,N)$ is also essentially trivial.
\end{proof}

This means that it would be possible dispense with the assumption that a skew bracoid is reduced for results in which the properties of a skew bracoid imply those of a subgroup of the holomorph. However, the assumption is required in the reverse as the converse does not hold in general: recall from \Cref{d2ncd} that $(D_{18}, C_3)$ does not contain a brace while the reduced form $(D_6,C_3)$ is almost classical. With this caveat, we can restrict our correspondence between skew bracoids and subgroups of the holomorph to the almost a brace case fairly cleanly as follows: 

\begin{prop}\label{alalhol}
Let $(G,N,\odot)$ be a skew bracoid and let $A=\lambda_\odot(G) \subseteq \Hol(N)$. The following are equivalent:
\begin{enumerate}[(i)]
    \item $A$ is of the form $R\rtimes B$ where $R$ is a regular subgroup of $\Hol(N)$ and $B$ is some subgroup of $\Aut(N)$; \label{holtosbac}
    \item the skew bracoid $(G,N,\odot)$ is almost a brace. \label{sbtoholac}
\end{enumerate}
\end{prop}

\begin{proof}
Suppose $A$ is of the form given in (\ref{holtosbac}). Elements of $\Hol(N)$ stabilise $e_N$ precisely when they have trivial $N$ component, so $S:=\Stab_G(e_N)$ is equal to the preimage $\lambda_\odot^{-1}(B)$. We also have that $H:=\lambda_\odot^{-1}(R)$ is a normal subgroup of $G$ which has trivial intersection with $S$, as $R$ does with $B$. Hence $G=HS\cong H\rtimes S$ and $(G,N,\odot)$ is almost a brace.

Suppose the skew bracoid $(G,N,\odot)$ is almost a brace with respect to $H$ a normal complement to $S=\Stab_G(e_N)$. Recall that this means $H$ is regular on $N$ so in view of \Cref{biggamma}, we know that $R:=\lambda_\odot(H)$ is regular on $N$. For $s\in S$ we have 
\[\lambda_\odot(s)=(s\odot e_N,\gamma(s))=(e_N,\gamma(s)),\] 
so $\lambda_\odot(S)=(e_N,\gamma(S))$, and we may take $B$ to be $\gamma(S)\subseteq \Aut(N)$. Since $G=HS\cong H\rtimes S$, its image under the injection $\lambda_\odot$ must have $\lambda_\odot(H)\lambda_\odot(S)=RB \cong R\rtimes B$. 
\end{proof}

In the almost classical case we have an even cleaner description.

\begin{prop}\label{alhol}
    Let $(G,N,\odot)$ be a skew bracoid and let $A=\lambda_\odot(G) \subseteq \Hol(N)$. The following are equivalent:
    \begin{enumerate}[(i)]
        \item $A$ is of the form $N\rtimes B$ for some subgroup $B$ of $\Aut(N)$;\label{holnb}
        \item $A$ contains $(N, \id)$;\label{ninhol}
        \item the skew bracoid $(G,N,\odot)$ is almost classical.\label{alcla}
\end{enumerate}
\end{prop}

\begin{proof}
First suppose that (\ref{holnb}) holds. Then $N$, taken as a subgroup of its holomorph, is certainly regular on itself, so $A$ is of the form given in \Cref{alalhol}. Hence $(G,N,\odot)$ is almost a brace with respect to $H=\lambda_\odot^{-1}(N)$. The sub-skew bracoid $(H,N)$ is essentially trivial since the operation in $H \cong N\subseteq \Hol(N)$ agrees with that in $N$ itself, thus (\ref{alcla}) holds.

Next suppose (\ref{alcla}) holds with respect to some $H$, a normal complement to $S=\Stab_G(e_N)$. For $h\in H$ we have $\lambda_\odot(h)=(h\odot e_N,\gamma(h))$, by \Cref{acgamma} we know that $\gamma(h)$ is trivial. We also know that $H$ is regular on $N$, so
\[\lambda_\odot(H)=\{(h\odot e_N,\id)\;|\; h\in H\} = (N,\id) \]
and (\ref{ninhol}) holds. 

Finally suppose (\ref{ninhol}) holds and consider $B=A \;\cap\; (e_N, \Aut(N))$, the subgroup of $A$ consisting of its purely automorphism elements. Let $(\eta, \alpha)\in A$; as $(N,\id)\subseteq A$, we know that $(e_N,\alpha)= (\eta^{-1},\id)(\eta, \alpha)$ is in $A$ and therefore in $B$. We can then write $(\eta, \alpha)$ as a product of an element of $N$, $(\eta,\id)$, and an element of $B$, $(e_N,\alpha)$. It is clear that $N$ and $B$ have trivial intersection, and we know $N$ is normal in $A$ as it is normal in $\Hol(N)$, so $A=NB\cong N\rtimes B$ and we have (\ref{holnb}).
\end{proof}

We see this behaviour borne out in our example.

\begin{eg}
    Consider the skew bracoid $(G,N)\cong (D_{2n}, C_n)$ from \Cref{d2ncd}. This has $\gamma$-function given by $\acts{\gamma(r^is^j)}\eta^k=\eta^{(-1)^jk}$ so, writing $\iota$ for inversion, \[\lambda_\odot(r^is^j)= (r^is^j\odot e_N,\gamma_{r^is^j}) = (\eta^i,\iota^j)\] and $\lambda_\odot(G) =N\rtimes \langle \iota \rangle$.
\end{eg}

With this holomorph description, we can give an enumeration result for almost classical skew bracoids.

\begin{cor} \label{enuhol}
    Let $N$ be a group. Up to isomorphism and equivalence, the number of almost classical skew bracoids with $N$ as their additive group is equal to the number of conjugacy classes of subgroups of $\Aut(N)$.
\end{cor}

\begin{proof}
    For each subgroup $B$ of $\Aut(N)$, we get a subgroup $A= N\rtimes B$ of $\Hol(N)$ that is transitive on $N$, which gives rise to a skew bracoid $(A,N)$ that is almost classical by \Cref{alhol}. Moreover, by \Cref{alhol} and the definition of equivalence \cite[Definition 2.17]{mltruoid}, every almost classical skew bracoid with additive group $N$ is equivalent to one of this form. Then from \cite[Corollary 4.16]{mltruoid} we know that two such skew bracoids, $(A,N)\cong(N\rtimes B,N)$ and $(A',N)\cong(N\rtimes B',N)$, are isomorphic if and only if there exists some $\theta\in \Aut(N)$ for which $A=\theta A' \theta^{-1}$; this occurs precisely when $B=\theta B'\theta^{-1}$ as $\theta(N)=N$ for all $\theta$. 
\end{proof}

\section{Connection with almost classical Hopf-Galois structures}\label{sechg}

In this section we confirm the connection between our definition and an existing concept in Hopf-Galois theory. We begin by reviewing the Hopf-Galois picture, starting long before the inception of skew bracoids with Greither and Pareigis theory. In their 1987 paper, Greither and Pareigis made the study of Hopf-Galois structures on finite separable extensions of fields a question of group theory \cite{gp}. Let $L/K$ be a such an extension with Galois closure $E$, write $G$ for $\Gal(E/K)$, $S$ for $\Gal(E/L)$ and $X$ for the coset space $G/S$. They show that there is a bijection between Hopf-Galois structures on $L/K$ and regular \textit{$G$-stable} subgroups $N$ of $\Perm(X)$; here $G$-stable is saying $N$ is normalised by $\lambda(G)$, left translation of cosets. They explicitly prove that every Hopf-Galois structure on $L/K$ is isomorphic to one of the form $E[N]^G$ for such an $N$, the so-called \textit{type} of the Hopf-Galois structure. 

As mentioned in \Cref{intro}, an extension is said to be almost classical if $S$ has a normal complement $H$ in $G$, so that $G$ decomposes into the semi-direct product $H\rtimes S$. While a Hopf-Galois structure is itself said to be \textit{almost classical} if it corresponds under the theorem of Greither and Pareigis to a subgroup of $\Perm(X)$ of the form \[\lambda(H)^{opp}:= \{\eta \in \Perm(X): \lambda(h)\eta=\eta\lambda(h) \text{ for all }h\in H\},\] for some normal complement $H$ of $S$. 

To give the correspondence between Hopf-Galois structures and skew bracoids, we recall that every skew bracoid is isomorphic to one of the form $(G,\cdot,X,\star,\odot)$ where $X=G/S$, $S=\Stab_G(e_X)$, the identity in the additive group is the identity coset $eS$, and $G$ acts on $X$ by left translation of cosets \cite[Example 4.9]{mltruoid}. For this section we work with skew bracoids using this presentation that are finite, meaning both the multiplicative group and the additive group are finite. With the notation as above, though allowing $E$ to be any Galois extension of $K$ containing $L$, \cite[Theorem 5.1]{mltruoid} gives a bijection between Hopf-Galois structures on $L/K$ and additive operations on $X$ that give a skew bracoid structure to $(G,X)$. (Note that the stabiliser $\Stab_G(e_X)$ in these skew bracoids coincides with the Galois group $\Gal(E/L)$.) Explicitly the Hopf-Galois structure associated to a skew bracoid $(G,X)$ is $E[X,\star]^{(G,\cdot)}$, where $G$ acts on $E$ by Galois automorphisms and on $X$ via the $\gamma$-function of the skew bracoid. The action of $E[X,\star]^{(G,\cdot)}$ on $L$ is given by \[
    \left(\sum_{\bar{g}\in X}c_{\bar{g}}\bar{g}\right)[t]= \sum_{\bar{g}\in X}c_{\bar{g}}\bar{g}(t) \text{   for all }t\in L \text{ \cite[Proposition 5.7]{mltruoid}}.
\] Note also that the skew bracoid is reduced (the action is faithful) if and only if the field $E$ is the Galois closure of $L/K$ \cite[Corollary 5.3]{mltruoid}.

Following immediately from this, we have:

\begin{cor}\label{skewonalmostext}
Suppose a skew bracoid $(G,X)$ corresponds to a Hopf-Galois structure on some separable extension of fields as above. The skew bracoid is almost a brace if and only if $S$ has a normal complement in $G$. If we assume additionally that $(G,X)$ is reduced, $(G,X)$ is almost a brace if and only if extension is almost classical. 
\end{cor}

We turn now to the almost classical skew bracoid.

\begin{rmk}
    Suppose we have an almost classical skew bracoid written in the form $(G,X)$ as above. Since we have fixed the action, the bijection $h\mapsto h\odot eS$ is more transparent: writing $\bar h$ for the coset $hS$ we have simply $h\mapsto \bar{h}$. Then \Cref{niceop} becomes
    \begin{equation}\label{niceopx}
        \overline{h_1} \star \overline{h_2} = \overline{h_1h_2}
    \end{equation}
    for all $h_1,h_2\in H$ and \Cref{niceinv} becomes
    \begin{equation}\label{niceinvx}
        (\overline{h})^{-1} = \overline{h^{-1}}
    \end{equation}
    for all $h\in H$. It must be stressed that these exceedingly clean relationships only necessarily hold using the coset representatives in $H$.
\end{rmk}

\begin{thrm}
\label{sbperm}
Let $L/K$ be a separable extension of fields as above, the additive operations $\star$ on $X$ under which $(G,X)$ is an almost classical skew bracoid are in bijective correspondence with the almost classical Hopf-Galois structures on $L/K$.
\end{thrm}

\begin{proof}
Due to the theorem of Greither and Pareigis \cite{gp}, it will be enough to show that there is a bijective correspondence between operations $\star$ as in the statement and subgroups of $\Perm(X)$ of the form $\lambda(H)^{opp}$ for some $H$ giving $G=HS\cong H\rtimes S$.

Let $(G,X)$ be an almost classical skew bracoid with additive operation $\star$, and $H$ be a normal complement to $S$ in $G$ for which $(H,X)$ is essentially trivial. As in \cite[Theorem 5.1]{mltruoid}, we consider the right regular representation $\rho_\star$ of $(X,\star)$ into $\Perm(X)$. Being sure to write our cosets using their representatives in $H$, for $h_1,h_2 \in H$ we then have 
\begin{align*}
\rho_\star(\overline{h_1})[\overline{h_2}]&= \overline{h_2} \star (\overline{ h_1})^{-1}  &&    \\
						&= \overline{h_2} \star \overline{h_1^{-1}}	  &&\text{by (\ref{niceinvx})}    \\
						&= \overline{h_2h_1^{-1}} &&\text{by (\ref{niceopx})}.
\end{align*}
We claim that $\rho_\star(X)$ coincides with $\lambda(H)^{opp}$ for this $H$. By the definition of $\lambda(H)^{opp}$ and the fact that $\lambda(h_1)\rho_\star(h_2)[h_3S]=(h_1h_3h_2^{-1})S=\rho_\star(h_2)\lambda(h_1)[h_3S]$ for all $h_1,h_2,h_3\in H$, we have $\rho_\star(X)\subseteq \lambda(H)^{opp}$ and a size argument gives equality.

Conversely, let $H$ be a normal complement to $S$ in $G$. Then $\lambda(H)$, equivalently $\lambda(H)^{opp}$, is regular on $X$ and $G$-stable since $H$ forms a set of coset representatives for $X$ and is normal in $G$. Define $\rho_H:H\to \Perm(X)$ by $\rho_H(h_1)[\overline{h_2}]=\overline{h_2h_1^{-1}}$ for all $h_1,h_2\in H$. Since $H$ forms a set of coset representatives for $X$, defining a permutation using the representatives in $H$ is enough to fully define it and since these representatives are unique there can be no ambiguity in its effect. Noting that $\rho_H(H)$ coincides with $\rho_\star(X)$ above, it therefore coincides with $\lambda(H)^{opp}$. 

The fact that the images of $\rho_\star$ and $\rho_H$ agree is enough to give the correspondence, but for completeness we define the operation $\star$ on $X$ coming from $\rho_H$. Recall from \cite[Theorem 5.1]{mltruoid} that for a subgroup $N\subseteq \Perm(X)$, we use the bijection $a:N\to X$ given by $\eta\mapsto \eta^{-1}[\bar e]$ to transfer the operation in $N$ onto $X$. In our case specifically, this is then $a(\rho_H(h))=\rho_H(h)^{-1}[\bar e]=\bar h$. For $h_1,h_2\in H$, we define
\begin{alignat*}{2}
\overline{h_1} \star \overline{h_2} &&{}:={}& a(a^{-1}(\overline{h_1})a^{-1}(\overline{ h_2}))  \\
					&&{}={}& a(\rho_H(h_1)\rho_H(h_2))											\\
					&&{}={}& \rho_H(h_2)^{-1}\rho_H(h_1)^{-1}[\bar e]							\\
					&&{}={}& \rho_H(h_2)^{-1}[\overline{h_1}]									\\
					&&{}={}& \overline{h_1h_2},
\end{alignat*}
so that $\star$ in $X$ is really $\cdot$ in $H$ as required.
\end{proof}

As discussed in \Cref{intro}, one of the central questions of Hopf-Galois theory is that of the surjectivity of the Hopf-Galois correspondence. The connection between skew braces and Hopf-Galois structures on Galois extensions has been incredibly useful to the study of this question in recent years, especially following the realignment of the correspondence due to Stefanello and Trappeniers \cite[Theorem 3.1]{sttwistpub}. This means that left ideals of the skew brace are in correspondence with $K$-Hopf sub-algebras of the associated Hopf-Galois structure and therefore the intermediate fields that occur in the image of the Hopf-Galois correspondence \cite[Theorem 3.9]{sttwistpub}. Using this, Stefanello and Trappeniers give many results including providing a characterisation of the extensions for which the Hopf-Galois correspondence is surjective for every structure they admit \cite[Theorem 4.24]{sttwistpub}. It is our hope that the skew bracoid may be a similarly powerful tool in the study of the Hopf-Galois correspondence in the separable case, with \cite[Theorem 5.9]{mltruoid} providing the analogous result on the alignment of substructures. As a prototype, we can now see a result of Greither and Pareigis as a consequence of \Cref{sbperm} and \Cref{hgcsb}. 

\begin{cor}\textbf{\cite[Theorem 5.2]{gp}}
Let $L/K$ be almost classically Galois. The Hopf-Galois correspondence is surjective for all almost classical Hopf-Galois structures on $L/K$.
\end{cor}

\begin{proof}
Let $(G,X)$ be an almost classical skew bracoid, corresponding to an almost classical Hopf-Galois structure on $L/K$.

From classical Galois theory we know that the intermediate fields of $L/K$ are in bijective correspondence with subgroups $G'$ of $G$ that contain $S$. By \cite[Theorem 5.9]{mltruoid}, these intermediate fields then occur in the image of the Hopf-Galois correspondence if and only if $G'\odot \bar e$ is a left ideal of $(G,X)$. \Cref{hgcsb} shows precisely that this holds for all such $G'$. 
\end{proof}

We note that our characterisation of almost classical in skew braciods is slightly broader than that implied by the concept in Hopf-Galois theory, firstly we allow the skew bracoid to be infinite, which is simply not covered by Greither and Pareigis theory; and secondly we do not assume that the skew bracoid is reduced or equivalently that $E$ is the Galois closure in particular. We take this relaxed definition because many of the results concerning almost classical extensions hold in this more general setting, including the induction construction of \Cref{inducesec}, which we may now see as a generalisation of the existing induced machinery of \cite{crvind} to a tower of merely separable extensions.

\begin{cor}
    Let $E/K$ be a Galois extension of fields with Galois group $G$. Suppose $G=HS\cong H\rtimes S$ and let $S'$ be a subgroup of $S$. Taking $L:=E^S$ and $L':=E^{S'}$ in the classical Galois sense, we have a tower of fields with $L/K$ and $L'/L$ not necessarily Galois, as in \Cref{figind}. Suppose we have a Hopf-Galois structure of type $N$ on $L/K$ and a Hopf-Galois structure of type $M$ on $L'/L$. Then there is a Hopf-Galois structure of type $N\times M$ on $L'/K$.
\end{cor}

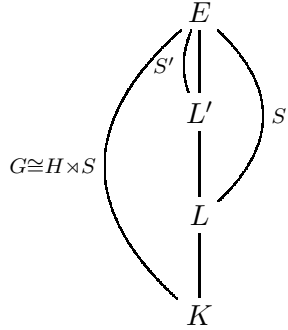
\begin{figure}[h]
    \centering
    \[\xymatrixcolsep{3pc} 
\xymatrixrowsep{2pc}
\xymatrix{
\ar@/_3pc/@{-}[ddd]_{G\cong H\rtimes S} E \ar@/_0.5pc/@{-}[d]_{S'} \ar@{-}[d] \ar@/^2pc/@{-}[dd]^{S} \\
L' \ar@{-}[d]\\
L \ar@{-}[d] \\
K }\]
    \caption{Field diagram for the induced set-up.}
    \label{figind}
\end{figure}

\begin{proof}
    In view of \cite[Theorem 5.1]{mltruoid} and \Cref{skewonalmostext}, this is the set-up of \Cref{induce}. Taking the resulting induced skew bracoid and translating back, we have a Hopf Galois structure of type $N\times M$ on $L'/K$.
\end{proof}

\subsection{Byott's Translation}\hfill \label{sectrans}

We noted in \Cref{intro} that under the current procedures, the route from Hopf-Galois structure to skew bracoid does not necessarily yield the same skew bracoid as the one obtained by first passing to the holomorph, even up to isomorphism. We conclude this section with a proposed realignment of the connection between the permutation and holomorph setting to account for this.

We begin by briefly reviewing the situation as it stands; we use the same notation as before and for simplicity we take $E$ to be the Galois closure. There are established correspondences between four objects, namely:
\begin{enumerate}[(i)]
    \item Hopf-Galois structures on $L/K$; \label{hgs}
    \item regular $G$-stable subgroups $N$ of $\Perm(X)$; \label{perm}
    \item transitive subgroups of $\Hol(N)\cong \Hol(X)$; \label{hol}
    \item and (reduced) skew bracoids $(G,X)$.\label{sboid}
\end{enumerate}
The relationship between (\ref{hgs}) and (\ref{perm}) is demonstrated by Greither and Pareigis in their foundational paper \cite{gp}. Following an observation of Childs in \cite{ch89}, Byott shows that there is a Hopf-Galois structure of type $N$ on $L/K$ if and only if there is a transitive embedding of $G$ into $\Hol(N)$ with the image of $S$ stabilising the identity \cite{btrans}. A review of this result can be found in \cite[Chapter 2 \S 3]{chetal} or \cite[\S 7]{chbook}, and the crucial chain of maps is summarised in \Cref{byo}. Building on the work of \cite{sttwistpub}, we see the correspondence between (\ref{hgs}) and (\ref{sboid}) in \cite{mltruoid}, showing the relationship between (\ref{sboid}) and (\ref{perm}) as part the proof. Separately, \cite{mltruoid} also shows the correspondence between (\ref{sboid}) and (\ref{hol}). 

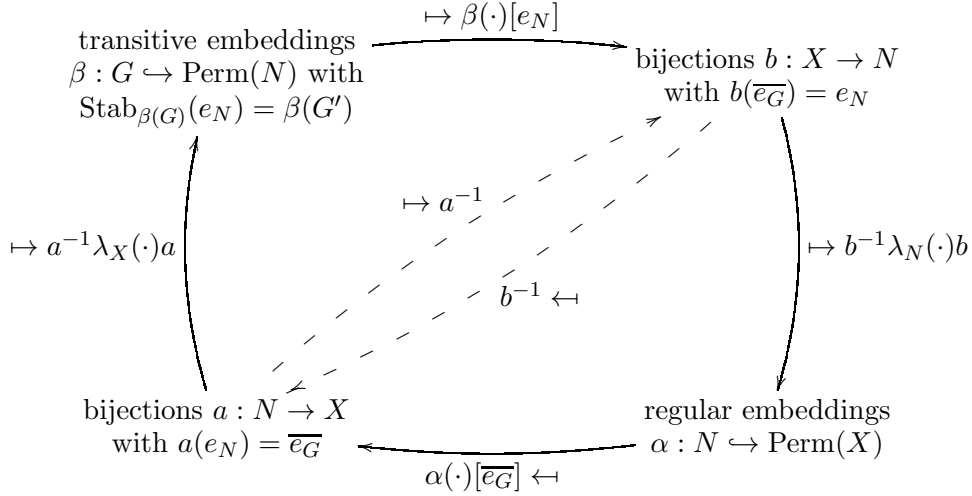
\begin{figure}[h]
\centering
\[\xymatrixcolsep{8pc} 
\xymatrixrowsep{8pc}
\renewcommand{\labelstyle}{\textstyle}
\xymatrix{
    \txt{transitive embeddings  \\$\beta:G\into \Perm(N)$ with \\$\Stab_{\beta(G)}(e_N)=\beta(G')$} 
    \ar@/^1pc/@{->}[r]^{\mapsto \beta(\cdot)[e_N]} 
    & \txt{bijections $b:X\to N$\\with $b(\overline{e_G})=e_N$} 
    \ar@/^1pc/@{-->}[dl]^{b^{-1}\mapsfrom}\ar@/^1pc/@{->}[d]^{\mapsto b^{-1}\lambda_N(\cdot)b}\\
    \ar@/^1pc/@{->}[u]^{\mapsto a^{-1}\lambda_X(\cdot)a} \ar@/^1pc/@{-->}[ur]^{\mapsto a^{-1}}
    \txt{bijections $a:N\to X$\\with $a(e_N)=\overline{e_G}$} 
    & \ar@/^1pc/@{->}[l]^{\alpha(\cdot)[\overline{e_G}]\mapsfrom} 
    \txt{regular embeddings \\$\alpha:N\into \Perm(X)$}
}\]
\caption{The existing maps, before specialiasing to the $G$-stable, holomorph case. \cite{btrans}}
\label{byo}
\end{figure}

To illustrate the issue consider the following example.

\begin{eg}
Take some Galois extension of fields, $E/K$, with non-abelian Galois group $G$. It is well known that such an extension admits a canonical non-classical structure corresponding to $\lambda(G)\subseteq \Perm(G)$. Stefanello and Trappeniers show that under their procedures this structure corresponds to the \textit{almost trivial skew brace} on $G$, $(G,\cdot^{opp},\cdot)$ \cite[Example 3.5]{sttwistpub}. 

Following the maps outlined in \Cref{byo}, noting that in this setting $N$ and $X$ both coincide with $G$, we find that $\alpha=\lambda :G\into \Perm(G)$ leads to $a=\id : G\to G$ which in turn gives $\beta=\lambda:G\into \Perm(G) $ and we recover $\lambda(G)$ in $\Hol(G)$. To find the corresponding skew brace we may transfer binary operation in $\lambda(G)\subseteq \Hol(G)$ onto $G$ itself using the bijection defined by evaluation at the identity. It is routine to check that this results in the trivial skew brace on $G$, $(G,\cdot,\cdot)$, rather than the almost trivial as before.
\end{eg} 

We see in the above example that the skew brace obtained via one route is the \textit{opposite} (see \cite{kochideal}) of that obtained via the other. Indeed, it is the central idea of \cite{sttwistpub} to associate a Hopf-Galois structure with the opposite skew brace to that in the original connection given in \cite[Appendix A]{svb}. It must be emphasised that the discrepancy we now identify is present in but irrelevant to \cite{sttwistpub}, since their correspondence, and the subsequent extension to the separable case in \cite{mltruoid}, intentionally bypasses both the permutation and holomorph setting. However, since we have investigated qualitative properties of skew bracoids and Hopf-Galois structures via the holomorph, we must reconcile this. To do so we propose a relabelling between (\ref{perm}) and (\ref{hol}). 

\subsubsection{The Relabelling}\hfill

The translation works using a series of maps between sets of embeddings and bijections, as shown in \Cref{byo}. Note that these results only hold as written using $E$ as the Galois closure. In this notation, we wish to align $\beta(G)$ with $\alpha(N)^{opp}$.

\begin{prop}
Let $\beta:G\into \Perm(N)$ be a suitable embedding and $\iota:N\to N$ denote the inversion map. Then, 
\begin{enumerate}[(i)]
    \item $\hat{\beta}$ given by $g\mapsto\iota \beta(g)\iota$ for all $g\in G$ is also a suitable embedding; \label{betahat}
    \item the corresponding $\hat{\alpha}$ has $\hat{\alpha}(N)=\alpha(N)^{opp}$.\label{alphahat}
\end{enumerate}
\end{prop}

\begin{proof}
We begin with (\ref{betahat}). We see that $\hat{\beta}$ is a homomorphism from the fact that $\beta$ is a homomorphism: $\iota\beta(gh)\iota=\iota\beta(g)\iota\iota\beta(h)\iota$ for all $g,h\in G$. Now $\hat{\beta}$ is injective since if $\iota\beta(g)\iota[\eta]=\eta$ for all $\eta\in N$, then $\beta(g)[\eta^{-1}]=\eta^{-1}$ for all $\eta\in N$ which implies $g=e_G$ by the injectivity of $\beta$. The transitivity of $\beta(G)$ gives $\iota\beta(G)\iota[e_N]=\iota[N]=N$ so $\hat{\beta}(G)$ is also transitive. Finally, for $g\in G$ we have $\hat{\beta}(g)[e_N]=e_N$ if and only if $\iota\beta(g)[e_N]=e_N$ or equivalently $\beta(g)[e_N]=e_N$, hence $\hat{\beta}(g)\in\Stab_{\hat{\beta}(G)}(e_N)$ if and only if $\beta(g)\in\Stab_{\beta(G)}(e_N)$, that is if and only if $g\in G'$.

For (\ref{alphahat}), first observe that the bijection $\hat{b}:X\to N$ coming from $\hat{\beta}$ relates to the analogous $b$ from $\beta$ via $\hat{b}(\bar{g})=\iota\beta(g)\iota[e_N]=\iota\beta(g)[e_N]=\iota b(\bar{g})$ for all $\bar{g}\in X$. We then have $\hat\alpha(\eta)=b^{-1}\iota\lambda_N(\eta)\iota b$ for all $\eta\in N$. Further, for all $\eta,\mu\in N$ we have
\begin{equation*}
    \iota\lambda_N(\eta)\iota[\mu] = \iota(\eta\mu^{-1})=\mu\eta^{-1}=\rho_N(\eta)[\mu].
\end{equation*}
so that  
\begin{equation*}
\hat{\alpha}(\eta)= b^{-1}\iota\lambda_N(\eta)\iota b = b^{-1} \rho_N(\eta) b.
\end{equation*}

Let $\eta,\mu\in N$, we then have
\begin{align*}
\alpha(\mu)\hat{\alpha}(\eta) &= b^{-1}\lambda_N(\mu)b b^{-1}\rho_N(\eta)b           \\
                                &= b^{-1}\lambda_N(\mu)\rho_N(\eta)b                \\
                                &= b^{-1}\rho_N(\eta)\lambda_N(\mu)b                \\
                                &= b^{-1}\rho_N(\eta)b b^{-1}\lambda_N(\mu)b         \\
                                &= \hat{\alpha}(\eta)\alpha(\mu).   
\end{align*}
Hence $\hat{\alpha}(\eta)$ and $\alpha(\mu)$ commute for all $\eta, \mu\in N$, and a size argument leads to ${\hat{\alpha}(N)=\alpha(N)^{opp}}$.
\end{proof}

The idea is then to adjust the maps so that $\beta$ corresponds with $\hat{\alpha}$ (and $\alpha$ corresponds with $\hat{\beta}$). The choices are somewhat arbitrary, we make these adjustments between $\alpha$ and its neighbours and give some justification for this after we show that it is indeed possible. For completeness and clarity we reproduce the remaining bijections from \cite{btrans}, note that we now dispense with the $\hat{\alpha}$ notation with the intention that what follows is entirely self contained.

\begin{prop}\textbf{\cite{btrans}}\label{byoskewprop}
There is a bijection between the following sets:
\begin{gather*}
\mathcal{A}=\{\text{regular embeddings } \alpha:N\into \Perm(X)\}, \\
\mathcal{B}=\{\text{transitive embeddings } \beta:G\into \Perm(N) \text{ such that } \beta(G')=\Stab_{\beta(G)}(e_N) \}. 
\end{gather*}
\end{prop}

\begin{proof}
Suppose we have a transitive embedding $\beta\in \mathcal{B}$. We may take $b:X\to N$ given by $b(\barg)=\beta(g)[e_N]$, this is well defined due to the condition on the stabiliser of $e_N$. Then $b$ is a bijection as $\beta(G)$ is transitive on $X$ and $|X|=|N|$, also $b(\overline{e_G})=\beta(\overline{e_G})[e_N]=e_N$ as $\beta$ is a homomorphism. From here, we construct $\alpha_b:N\to \Perm(X)$ defined by $\alpha_b(\eta)=b^{-1}\rho_N(\eta)b$ and claim that this $\alpha_b$ is in $\mathcal{A}$. First note that for all $\eta,\mu\in N$ we have
\begin{align}\label{alphom}
\begin{split}
\alpha(\eta\mu) &= b^{-1}\rho_N(\eta\mu)b               \\
                &= b^{-1}\rho_N(\eta)\rho_N(\mu)a^{-1}       \\
                &= b^{-1}\rho_N(\eta)bb^{-1}\rho_N(\mu)b\\
                &= \alpha(\eta)\alpha(\mu),
\end{split}
\end{align}
so $\alpha$ is a homomorphism. Then $\alpha(\eta)= \id$ means $\alpha(\eta)[\barg]= \barg$ for all $\barg\in X$, and
\begin{equation}\label{alpinj}
\begin{alignedat}{4}
   & &{} \alpha(\eta)[\barg]&= \barg     {}& \text{ for all }\barg\in X \\ 
   &\implies&{}\;\;\; b^{-1}\rho_N(\eta)b[\barg]&= \barg    {}& \text{ for all }\barg\in X \\
   &\implies&{} \rho_N(\eta)b[\barg]&= b[\barg]  {}& \;\;\;\text{ for all }\barg\in X \\
   &\implies&{} \eta&= e_N, {}&
\end{alignedat}
\end{equation}
due to the bijectivity of $b$ and the injectivity of $\rho_N$, so $\alpha$ is an embedding. Finally, $\alpha(N)$ is indeed regular on $X$, since $b$ is a bijection and $\rho_N(N)$ is regular on $N$; therefore $\alpha_b\in \mathcal{A}$.

Conversely, if we begin with $\alpha\in \mathcal{A}$, since $\alpha(N)$ is regular on $X$ we can define a bijection $a:N\to X$ given by $a(\eta)=\alpha\iota(\eta)[\overline{e_G}]$. Note that $a(e_N)=\alpha\iota(e_N)[\overline{e_G}]=\overline{e_G}$. With this we may construct $\beta_a:G\to \Perm(N)$ given by $\beta_a(g)=a^{-1}\lambda_X(g)a$ for all $g\in G$, we claim $\beta_a\in \mathcal{B}$. The fact that $\beta_a$ is a homomorphism follows similarly to (\ref{alphom}), and, as we are working with the Galois closure, $\lambda_X$ is injective \cite[Lemma 6.6]{chbook} so $\beta_a$ is injective by a similar argument to (\ref{alpinj}). Hence $\beta_a$ is an embedding. We have that $\beta_a(G)$ is transitive on $N$ since $a$ is a bijection and $\lambda_X(G)$ is transitive on $X$. For the condition on the stabiliser suppose $\beta_a(g)[e_N]=e_N$, this is saying 
\begin{alignat*}{3}
& & \;\;\;a^{-1}\lambda_X(g)a[e_N]&=e_N    \\
&\iff&{} a^{-1}\lambda_X(g)[\overline{e_G}]&=e_N     \\
&\iff&{} a^{-1}[\barg]  &= e_N              \\
&\iff&{} \barg&=\overline{e_G},
\end{alignat*}
i.e. $g\in G'$, this means $\Stab_{\beta_a(G)}(e_N)=\beta_a(G')$ and $\beta_a \in \mathcal{B}$.

It remains to show that these procedures are mutually inverse. Let $\alpha\in \mathcal{A}$ and construct $a:N\to X$ and then $\beta_a\in \mathcal{B}$ as above. Consider the subsequent $b:X\to N$ coming from $\beta_a$, this is given by $b(\barg)=\beta(\barg)[e_N]=a^{-1}\lambda_X(\barg)a[e_N]=a^{-1}(\barg)$ for all $\barg\in X$. To show that $\alpha=\alpha_b$, we verify that $b\alpha(\eta) b^{-1} = \rho_N(\eta)$; for $\eta,\mu\in N$ we have,
\begin{align*}
b\alpha(\eta) b^{-1}[\mu]&= a^{-1}\alpha(\eta) a[\mu]                       \\
                        &= a^{-1}\alpha(\eta)\alpha\iota(\mu)[\overline{e_G}]      \\
                        &= a^{-1}\alpha(\eta\mu^{-1})[\overline{e_G}]              \\
                        &= a^{-1}\alpha\iota(\mu\eta^{-1})[\overline{e_G}]         \\
                        &= a^{-1}a(\mu\eta^{-1})                            \\
                        &= \mu\eta^{-1}                                     \\
                        &= \rho_N(\eta)[\mu].
\end{align*}
Going from $\alpha$ to $\alpha_b$ via $\beta_a$ therefore amounts to the identity on $\mathcal{A}$. The analogous result on $\mathcal{B}$ follows similarly.
\end{proof}

The maps used in the above proof are summarised in \Cref{byoskewfig}.

\begin{figure}[h]
\centering
\[\xymatrixcolsep{8pc} 
\xymatrixrowsep{8pc}
\renewcommand{\labelstyle}{\textstyle}
\xymatrix{
    \txt{transitive embeddings  \\$\beta:G\into \Perm(N)$ with \\$\Stab_{\beta(G)}(e_N)=\beta(G')$} 
    \ar@/^1pc/@{->}[r]^{\mapsto \beta(\cdot)[e_N]} 
    & \txt{bijections $b:X\to N$\\with $b(\overline{e_G})=e_N$} 
    \ar@/^1pc/@{-->}[dl]^{b^{-1}\mapsfrom}\ar@/^1pc/@{->}[d]^{\mapsto b^{-1}\rho_N(\cdot)b}\\
    \ar@/^1pc/@{->}[u]^{\mapsto a^{-1}\lambda_X(\cdot)a} \ar@/^1pc/@{-->}[ur]^{\mapsto a^{-1}}
    \txt{bijections $a:N\to X$\\with $a(e_N)=\overline{e_G}$} 
    & \ar@/^1pc/@{->}[l]^{\alpha\iota(\cdot)[\overline{e_G}]\mapsfrom} 
    \txt{regular embeddings \\$\alpha:N\into \Perm(X)$}
}\]
\caption{The Revised Maps.}
\label{byoskewfig}
\end{figure}
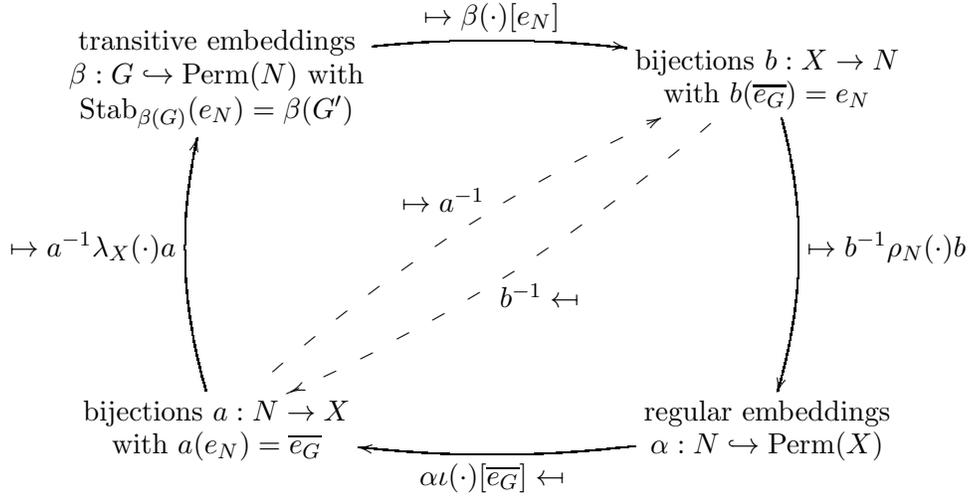

\begin{rmk}
To justify our choice of adjustment we note that this aligns with the $``a"$ used in \cite[Theorem 5.1]{mltruoid} to transfer the operation in $N$ onto $X$. In this setting $N\subseteq \Perm(X)$, so $\alpha$ is trivial and $a$ is given by $a(\eta)=\eta^{-1}[\bar e]$. Taking the other modification between $b$ and $\alpha$, as opposed to between $\beta$ and $b$, allows for the maintenance of the neat relationship between $a$ and $b$.
\end{rmk}

As in the traditional translation, the bijection between $\mathcal{A}$ and $\mathcal{B}$ restricts to a bijection between those $\alpha$ for which $\alpha(N)$ is $G$-stable and those $\beta$ for which $\beta(G)\subseteq \Hol(N)$.

\begin{prop}
Using the maps set out in the proof of \Cref{byoskewprop}, we have that $\alpha(N)$ is $G$-stable if and only if $\beta(G)\subseteq \Hol(N)$.
\end{prop}

\begin{proof}
Note that, 
\begin{equation*}
\alpha(\eta)= b^{-1}\rho_N(\eta)b = a \rho_N(\eta) a^{-1}.
\end{equation*}

For $g\in G$ and $\eta,\mu\in N$,
\begin{align*}
&&\lambda_X(g) \alpha(\eta)\lambda_X(g^{-1}) &= \alpha(\mu)                             \\
&\iff& a^{-1}\lambda_X(g)\alpha(\eta)\lambda_X(g^{-1}) a &=  a^{-1}\alpha(\mu)a         \\
&\iff& a^{-1}\lambda_X(g)a\rho_N(\eta)a^{-1}\lambda_X(g^{-1}) a &=  a^{-1}a\rho_N(\mu)a^{-1}a  \\
&\iff& \beta(g)\rho_N(\eta)\beta(g)^{-1} &= \rho_N(\mu), 
\end{align*}
so that $ \alpha(N)$ is $G$-stable if and only if $\beta(G)\subseteq \Norm_{\Perm(N)}(\rho_N(N))$, i.e. if and only if ${\beta(G)\subseteq \Hol(N)}$.
\end{proof}

\section{Almost Classical Solutions}
\label{secsol}

In \cite{ckmtybe}, we saw that skew bracoids that contain a brace can be used to produce solutions to the set theoretic Yang-Baxter equation, via their connection with semi-braces. As we have seen, to be almost a brace or almost classical, a skew bracoid must contain a brace and then satisfy additional conditions. In this final section, we consider how these properties manifest in the resulting solutions. But first, we review the procedure to obtain a solution from a skew bracoid.

Following \cite{ckmtybe}, we work with skew bracoids that contain a brace of the form $(G,H)$, with $H\subseteq G$ a complement to $S=\Stab_G(e)$. Such skew bracoids may be constructed from arbitrary skew bracoids that contain a brace via the bijection between the chosen complement and the additive group, see \cite[Section 2]{ckmtybe} for details. To produce a solution we then use the $\gamma$-function associated with the skew bracoid $(G,H)$ to define $\sigma:G\to \Map(G)$ and $\tau:G\to \Map(G)$ taking $g\mapsto \sigma_g$ and $\tau \mapsto \tau_g$ by
\begin{align*}
    \sigma_{g_1}(g_2)&= \gamma_{g_1}(g_2\odot e),   \\
    \tau_{g_2}(g_1)&= \sigma_{g_1}(g_2)^{-1}g_1g_2
\end{align*}
for all $g_1,g_2\in G$. This $\sigma$ and $\tau$ then give rise to a solution on $G$ given by $\bm{r}(g_1,g_2)= (\sigma_{g_1}(g_2),\tau_{g_2}(g_1))$ by \cite[Proposition 4.2]{ckmtybe}. Such solutions are \textit{right non-degenerate}, meaning $\tau_g$ is bijective for all $g\in G$, but only \textit{left non-degenerate}, $\sigma_g$ is bijective for all $g\in G$, if the skew bracoid is essentially a skew brace \cite[Section 4]{ckmtybe}.

We know that $(G,H)$ contains the brace $(H,H)$ if and only if $G$ has the exact factorisation $HS$ where $S=\Stab_G(e)$. Given this, it is intuitive that the resulting solution is isomorphic to a matched product, see \cite{ccsmatch}, of the solution coming from the skew brace on $H$ and a left-degenerate piece on $S$ \cite{ckmtybe}. We may then expect that strengthening the exact factorisation in $G$ to a semi-direct product, i.e. going from containing a brace to being almost a brace, strengthens the matched product of solutions to a semi-direct product.

\begin{prop}
    Let $(G,H)$ be a skew bracoid that is almost a brace $(H,H)$, write $S$ for $\Stab_G(e_N)$. The solution on $G$ is isomorphic to a semi-direct product of a solution on $H$ and one on $S$.
\end{prop}
    
\begin{proof}
    The matched product we refer to is seen through that in the associated semi-brace, note that the relevant exact factorisation in $G$ is fixed in the transformation between skew bracoid and semi-brace \cite[Theorem 3.5]{ckmtybe}. Taking the $\alpha:S\to \Perm(H)$ and $\beta:H\to \Perm(S)$ from \cite[Theorem 10]{ccsmatch} with the relationship between our $\sigma$ and $\tau$ and those coming from the semi-brace \cite[Proposition 4.1]{ckmtybe}, we know that the solution is isomorphic to a matched product via $\alpha$ and $\beta$ defined by $\alpha_h(s)=(\tau_{h^{-1}}(s^{-1}))^{-1}$ and $\beta_s(h)=\sigma_s(h)$ for all $h\in H$ and all $s\in S$. Then for $s\in S$ and $h\in H$ we have,
    \begin{align}\label{beta}
        \begin{split}
        \beta_s(h)&=\sigma_s(h)                          \\
                &= (s\odot e)^{-1}(s\odot (h\odot e))    \\
                &= sh\odot e                             \\
                &= shs^{-1},
        \end{split}
    \end{align} 
    and
    \begin{align*}
        \alpha_h(s)&=(\tau_{h^{-1}}(s^{-1}))^{-1}                  & \\
                &=(\sigma_{s^{-1}}(h^{-1})^{-1}s^{-1}h^{-1})^{-1}  & \\
                &=((s^{-1}h^{-1}s)^{-1}s^{-1}h^{-1})^{-1}          & \text{by (\ref{beta})} \\
                &=(s^{-1}hss^{-1}h^{-1})^{-1}                      & \\
                &=s.                                               & 
    \end{align*}
    We see that the actions $\alpha$ and $\beta$ are transparently the actions of $S$ on $H$ and $H$ on $S$ within $G$, in particular $\alpha$ is trivial, so the matched product is in fact a semi-direct product.
\end{proof}

We know from \cite[Proposition 7 (1)]{ccssemi}, seen through the correspondence \cite[Theorem 3.5]{ckmtybe}, that $\tau_g(S)=S$ for all $g\in G$, so we may restrict $\tau$ to a map $\tau|_S$ between $G$ and $\Map(S)$. We may then read a result of Castelli as saying that a skew bracoid $(G,H)$ is almost a brace with respect to $H$ if and only if $\ker(\tau|_S)=H$ \cite[Theorem 3.3]{castacsemi}. We recall that going from almost a brace to almost classical, we specify that the skew brace $(H,H)$ is trivial, which amounts to $H\subseteq\ker(\sigma|_H)$ by \Cref{gamker}. In this case, the solution on $H$ is the one coming from the trivial skew brace on $H$, or indeed the group $H$. Moreover, with our description of the $\gamma$-function of such skew bracoids from \Cref{acgamma}, noting that $s_1h_2s_1^{-1}\in H$ for any $h_2\in H$, we immediately obtain a complete description of the solution on $G$.

\begin{cor}\label{alclasol}
    Let $(G,H)$ be a skew bracoid that is almost classical with respect to $H$. There is a solution on $G$ given by
    \[ \bm{r}(h_1s_1,h_2s_2)= (s_1h_2s_1^{-1},s_1h_2^{-1}s_1^{-1}h_1s_1h_2s_2)\]
    for all $h_1,h_2\in H$ and all $s_1,s_2\in S$.
\end{cor}

We can consider our induced construction, \Cref{respind}, from a solution perspective. Given a skew bracoid that is almost a brace, we can use a skew bracoid that contains a brace on its stabiliser to lessen the degenerate piece. In the extreme case of taking a skew brace structure on this stabiliser, we recover a non-degenerate solution.

In a similar way to the proof of \Cref{enuhol}, we note that skew bracoids that contain a brace, and therefore solutions, can be constructed by taking the subgroup generated by a regular subgroup of the holomorph and any subgroup of the automorphism group. However, we cannot obtain a generic enumeration result like \Cref{enuhol} in the case of skew bracoids that contain or are almost a brace. Matters are complicated by the fact that these regular subgroups are not in general invariant under conjugation by automorphisms, and moreover, distinct regular subgroups could generate the same transitive subgroup when taken with carelessly chosen subgroups of $\Aut(N)$.

To conclude, we exploit the behaviour we saw in \Cref{d2ncd} - that a skew bracoid may contain multiple viable complements to the stabiliser, leading the skew bracoid to be almost classical with respect to one complement while merely contain a brace due to another. This allows us to draw solutions that are distinct in a meaningful way from a single skew bracoid.

\begin{eg}
Within our family of skew bracoids $(G,N)\cong (D_{2n},C_d)$ we focus on the reduced case in which $d=n$, and first take the complement $R=\langle r\rangle$ to $S=\langle s\rangle$. To write $(G,N)$ as $(G,R)$ we need only relabel $N=\langle \eta \rangle$ with $R$ in the obvious way. Since $(G,R)$ is almost classical with respect to $R$ we may employ \Cref{alclasol} to see that $G$ with
\begin{align*}
    \bm{r}(r^is^j,r^ks^\ell)&=(s^jr^ks^{-j},s^jr^{-k} s^{-j}                                    r^is^jr^ks^\ell)     \\
        &= (r^{(-1)^jk},r^{i}s^{j+\ell})
\end{align*}
is a solution.

When $n$ is even, the subgroup $H=\langle r^2,rs\rangle$ is another complement to $S$, by which $(G,N)$ is almost a brace. To find the resulting solution we to transfer the operation from $N$ onto $H$ using the inverse of the evaluation map $h\mapsto h\odot e_N$, which we denote $b$.
Note that \[b(\eta)=rs,\;\; b(\eta^2)=r^2, \;\;b(\eta^3)=r^3s, \;\; b(\eta^4)=r^4,\; \ldots\] so we may identify $H$ with a subgroup of $C_n\times C_2$ with $rs$ as a generator. The action of $G$ on $H$ is then $r^is^j\odot (rs)^k=(rs)^{i+(-1)^jk}$, morally as before.

Using $b:H\to G$ to emphasise which group structure is at work we then have,
\begin{align*}
\sigma_{r^is^j}(r^ks^\ell) &= b(\gamma_{r^is^j}((rs)^k))       \\
        &= b((rs)^{(-1)^jk})       \\
        &= b(r^{(-1)^jk}s^{(-1)^jk})\\
        &= r^{(-1)^jk}s^k,  \\ 
\tau_{r^ks^\ell}(r^is^j) &= \sigma_{r^is^j}(r^ks^\ell)^{-1}r^is^jr^ks^\ell   \\
        &= s^kr^{-(-1)^jk}r^{i+(-1)^jk}s^{j+\ell}     \\
        &= r^{(-1)^ki}s^{j+k+\ell}.
\end{align*}
Hence $G$ with $\textbf{r}(r^is^j,r^ks^\ell)=(r^{(-1)^jk}s^k,r^{(-1)^ki}s^{j+k+\ell})$ is also a solution.   
\end{eg}

\bibliography{refs}{}
\bibliographystyle{amsalpha}
\end{document}